\documentclass{tac}

\usepackage{xy}
\usepackage[centertags]{amsmath}

\xyoption{all}

\input diagxy

\usepackage[colorlinks=true]{hyperref}
\hypersetup{allcolors=[rgb]{0.1,0.1,0.4}}

\author{M. Golshani and A. Shiralinasab Langari}

\thanks{%
 The first author's research has been supported by a grant from
  IPM (No. 1400030417). The
	second author's research  is partially supported by IPM.}

\address{School of Mathematics, Institute for Research in Fundamental Sciences (IPM)\\
 Tehran-Iran P.O. Box:
19395-5746\\[5pt]
 Department of Mathematics, Shahid Bahonar University of Kerman\\Kerman, Iran\\
}

\title[Language of a topos via spans]{Representing the language of a topos as a quotient of the category of spans}

\copyrightyear{2025}


\eaddress{golshani.m@gmail.com\CR shiralinasab@gmail.com}

\newtheorem{theorem}{Theorem}

\newtheoremrm{rem}{Remark}

\mathrmdef{Hom}
\mathbfdef{Set}

\def\mc#1{\mathcal {#1}}
\def\C{\mc C}
\def\A{\mc A}
\def\B{\mc B}
\def\D{\mc D}

\def\I{\mc I}

\def\X{\mc X}
\def\Y{\mc Y}
\def\T{\mc T}
\def\Z{\mc Z}
\def\E{\mc E}
\def\M{\mc M}

\def\W{\mc W}

\begin{document}

\maketitle
\begin{abstract}
We use quotients of span categories to introduce the language of a topos. We also introduce the notion of logical relation and study the quotients of span categories derived from them. As an application we show that the category of Boolean toposes is a reflective subcategory of the category of toposes, when the morphisms are logical functors.
 \end{abstract}

\section{Introduction}

The Mitchell-B\'enabou language \cite{macmor} is a well-known form of the internal language of an elementary topos. In this approach, types are interpreted as objects of the topos, and variables are interpreted as identity morphisms \( 1:A \rightarrow A \). More generally, terms of type \( A \) in variables \( x_i \) of types \( X_i \) are interpreted as morphisms from the product \( \prod X_i \to A \). Formulas of the language are therefore identified with morphisms into the subobject classifier \( \Omega \).

A different but related approach is introduced in \cite{lambscot}, where variables are treated as \emph{indeterminate morphisms}. Given an object \( A \) in a topos \( \mathcal{T} \), a new category \( \mathcal{T}[x] \) is constructed by freely adjoining a morphism \( x:1 \to A \) to \( \mathcal{T} \). This is achieved by forming the free category generated by the graph obtained from the underlying graph of \( \mathcal{T} \) by adjoining such a morphism and closing under finite limits. Equivalently, this can be described as the Kleisli category of a cotriple \( (S_A, \epsilon_A, \delta_A) \), where \( S_A(X) = A \times X \), \( \epsilon_A(X) = \pi_X \), and \( \delta_A(X) = \langle \pi_A, 1_{A \times X} \rangle \).

However, this construction deals with one indeterminate at a time, and lacks a unified environment for reasoning with multiple variables. In this paper, we extend this framework by constructing a category where \emph{all indeterminate morphisms} are adjoined simultaneously. Our construction uses categories of spans and their quotients to provide such a setting.

\begin{itemize}
    \item We define, for each object \( A \) in a cartesian category \( \mathcal C \), a stable system \( \A \) and form a quotient category \( \mathsf{Span}_\A(\mathcal C, \A) \), in which a canonical morphism \( x = [!_A, 1_A]_\A:1 \to A \) plays the role of the indeterminate morphism.
    \item We present a quotient category of spans
\[
\mathsf{Span}_\Pi(\mathcal C,\Pi),
\]
which universally incorporates \emph{all} indeterminate morphisms.  Moreover, if \(\mathcal C\) is cartesian closed, then \(\mathsf{Span}_\Pi(\mathcal C,\Pi)\) is cartesian closed as well.
\end{itemize}

The category \( \mathsf{Span}_\Pi(\C, \Pi) \) provides a canonical setting for interpreting terms, formulas, and logical connectives in an internal manner. In this paper, we develop a formulation of the internal language of a topos \( \T \) within the structured environment of \( \mathsf{Span}_\Pi(\T, \Pi) \), where all variables are introduced simultaneously. This unified framework enables a coherent representation of the internal language in which variables and logical constructs coexist as morphisms of a single category.

This paper also investigates conditions under which a quotient category of spans \( \mathsf{Span}_\sim(\T) \) forms a power allegory, ensuring that \( \mathsf{Map}(\mathsf{Span}_\sim(\T)) \) is a topos. Leveraging this framework, we construct, in a universal manner, a Boolean topos associated to each elementary topos. As a consequence, we show that the category of Boolean toposes forms a reflective subcategory of the category of toposes, when morphisms are taken to be logical functors.

\section{Preliminaries}

We recall some definitions and preliminaries about \textit{span categories}. For more details, see \cite{hoshitho} and \cite{hsty}.

We consider \emph{categories equipped with a stable system of morphisms}; that is, pairs $(\mathcal{C}, \mathcal{S})$ where $\mathcal{C}$ is a category and $\mathcal{S}$ is a collection of morphisms in $\mathcal{C}$ satisfying the following properties:
\begin{itemize}
	\item $\mathcal{S}$ contains all isomorphisms in $\mathcal{C}$ and is closed under composition;
	\item pullbacks of $\mathcal{S}$-morphisms along arbitrary morphisms exist in $\mathcal{C}$ and belong to $\mathcal{S}$.
\end{itemize}

For objects $A, B$ in $\mathcal{C}$, a \emph{span} $(s, f)$ with domain $A$ and codomain $B$ consists of a diagram
\[
\xymatrix{
A & D \ar[l]_s \ar[r]^f & B
}
\]
where $s \in \mathcal{S}$ and $f$ is a morphism in $\mathcal{C}$.

Given another stable system $\mathcal{F}$, we define a morphism $x : (s, f) \to (s', f')$ with $x \in \mathcal{F}$ if the following diagram commutes:
\[
\xymatrix{
& & D \ar[lld]_s \ar[dd]_x \ar[rrd]^f & & \\
A & & & & B \\
& & D' \ar[llu]^{s'} \ar[rru]_{f'} & &
}
\]

If such a morphism $x$ exists, we write $(s, f) \leq_{\mathcal{F}} (s', f')$. The equivalence relation generated by $\leq_{\mathcal{F}}$ is denoted by $\sim_{\mathcal{F}}$.

We define the quotient category of spans $\mathsf{Span}_{\mathcal{F}}(\mathcal{C}, \mathcal{S})$, where:
\begin{itemize}
    \item Objects are the same as those of $\mathcal{C}$;
    \item Morphisms are equivalence classes $[s, f]_{\sim_{\mathcal{F}}}$ of spans under $\sim_{\mathcal{F}}$.
\end{itemize}

Composition of morphisms $[s, f]_{\sim_{\mathcal{F}}} : A \to B$ and $[t, g]_{\sim_{\mathcal{F}}} : B \to C$ is defined as $[s t', g f']_{\sim_{\mathcal{F}}}$, as in the following diagram:
\[
\xymatrix{
&& P \ar[ld]_{t'} \ar[rd]^{f'} && \\
& D \ar[ld]_s \ar[rd]^f \ar@{}[rr]|{\text{pb}} && E \ar[ld]_t \ar[rd]^g & \\
A && B && C
}
\]
This composition is well-defined. For simplicity, we write $[s, f]_{\mathcal{F}}$ instead of $[s, f]_{\sim_{\mathcal{F}}}$.

In the case where $\mathcal{F} = \mathcal{I}$ is the class of isomorphisms, the category $\mathsf{Span}_{\mathcal{I}}(\mathcal{C}, \mathcal{S})$ is the ordinary category of spans. In this case, we simply write $[s, f]$ for morphisms.

We now state a useful lemma about the equivalence relation $\sim_{\mathcal{F}}$:
\begin{lemma}[\cite{hsty}]\label{describe for stab class rel}
Let $\mathcal{F}$ be a stable system. Then $(s, f) \sim_{\mathcal{F}} (s', f')$ if and only if there exist $p, q \in \mathcal{F}$ such that the following diagram commutes:
\[
\xymatrix{
&& D \ar[lld]_s \ar[rrd]^f && \\
A && P \ar[d]_q \ar[u]^p && B \\
&& D' \ar[ull]^{s'} \ar[urr]_{f'} &&
}
\]
\end{lemma}

To further generalize the relation $\sim_{\mathcal{F}}$, we introduce the notion of a \emph{compatible relation} on $\mathsf{Span}(\mathcal{C}, \mathcal{S})$, which is a relation on spans satisfying:
\begin{itemize}
	\item only spans with the same domain and codomain may be related;
	\item vertically isomorphic spans are related;
	\item the equivalence relation defines a congruence on the category, that is, horizontal composition from either side preserves the relation.
\end{itemize}

For such a compatible equivalence relation $\sim$, we write the equivalence class of a span $(s, f)$ as $[s, f]_{\sim}$, or simply $[s, f]$ when the context makes it clear which relation is meant. The corresponding quotient category is denoted by
\[
\mathsf{Span}_{\sim}(\mathcal{C}, \mathcal{S}).
\]

\section{Adding indeterminate arrows}

Throughout this section, let $\mathcal{C}$ be a cartesian category.
As in \cite{lambscot}, for an object $A \in \mathcal{C}$, we aim to add an indeterminate morphism $x: 1 \to A$ to $\mathcal{C}$ in a universal way. To achieve this, we define a stable system $\mathcal{A}$ associated with the object $A$ and consider a quotient category of $\mathcal{A}$-spans as a setting where $x:1 \to A$ naturally lives.
For an object $A$ in $\mathcal{C}$, define the following class:
\[
\mathcal{A} = \{ \pi: A^n \times B \to B \mid \pi \text{ is a projection} \}.
\]

\begin{lemma}
For every object $A \in \mathcal{C}$, the class
\[
\mathcal{A} = \{ \pi: A^n \times B \to B \mid \pi \text{ is a projection} \}
\]
is a stable system.
\end{lemma}

\proof
For $n = 0$, we have $A^n = 1$, so $\mathcal{A}$ contains isomorphisms. Closure under composition and stability under pullbacks are straightforward.
\endproof

Using $\mathcal{A}$, we define the quotient category of spans:
\[
\mathsf{Span}_{\mathcal{A}}(\mathcal{C}, \mathcal{A}).
\]

\begin{proposition}
The map $\mathbf{Q} : \mathcal{C} \to \mathsf{Span}_{\mathcal{A}}(\mathcal{C}, \mathcal{A})$ sending a morphism $f$ to $[1, f]_{\mathcal{A}}$ is a functor. Furthermore, if there exists a morphism $1 \to A$, then $\mathbf{Q}$ is faithful.
\end{proposition}

\proof
It is clear that $\mathbf{Q}$ defines a functor. To prove faithfulness, suppose $[1, f]_{\mathcal{A}} = [1, g]_{\mathcal{A}}$ for morphisms $f, g: B \to C$ in $\mathcal{C}$. By Lemma~\ref{describe for stab class rel}, there exist morphisms $p, q \in \mathcal{A}$ such that the following diagram commutes:
\[
\xymatrix{
&& B \ar[lld]_{1} \ar[rrd]^{f} && \\
B && A^n \times B \ar[d]_q \ar[u]^p && C \\
&& B \ar[ull]^{1} \ar[urr]_{g} &&
}
\]
This implies $p = q$. Since there is a morphism $1 \to A$, the morphism $p$ is an epimorphism. Therefore, $f = g$, and so $\mathbf{Q}$ is faithful.
\endproof

\begin{theorem}
The functor $\mathbf{Q} : \mathcal{C} \to \mathsf{Span}_{\mathcal{A}}(\mathcal{C}, \mathcal{A})$ preserves finite products.
\end{theorem}

\proof
We show that
\(
\xymatrix{
B && B \times C \ar[ll]_{[1, \pi_B]_{\mathcal{A}}} \ar[rr]^{[1, \pi_C]_{\mathcal{A}}} && C
}
\)
is a product in $\mathsf{Span}_{\mathcal{A}}(\mathcal{C}, \mathcal{A})$, where $\pi_B$ and $\pi_C$ are the projections in $\mathcal{C}$.
Let
\(
\xymatrix{
B && D \ar[ll]_{[d_1, f]_{\mathcal{A}}} \ar[rr]^{[d_2, g]_{\mathcal{A}}} && C
}
\)
be a span, where $[d_1, f]_{\mathcal{A}}$ and $[d_2, g]_{\mathcal{A}}$ are represented by the diagrams in \(\C\):
\[
\xymatrix{
&& A^n \times D \ar[lld]_{d_1} \ar[rrd]^f && \\
D && && B
}
\quad
\xymatrix{
&& A^m \times D \ar[lld]_{d_2} \ar[rrd]^g && \\
D && && C
}
\]

Assuming $m \leq n$, there exists a projection $\pi: A^n \times D \to A^m \times D$. Then,
\[
[d_2, g]_{\mathcal{A}} = [d_2 \pi, g \pi]_{\mathcal{A}}.
\]
Since both $d_1$ and $d_2 \pi$ are projections from $A^n \times D$ to $D$, we can assume $n = m$ and $d_1 = d_2$. Let $d = d_1$ and $h = \langle f, g \rangle$. Then,
\[
[1, \pi_B]_{\mathcal{A}} \circ [d, h]_{\mathcal{A}} = [d, f]_{\mathcal{A}}, \quad [1, \pi_C]_{\mathcal{A}} \circ [d, h]_{\mathcal{A}} = [d, g]_{\mathcal{A}}.
\]

To prove uniqueness, suppose $[e, k]_{\mathcal{A}}$ is another morphism such that:
\[
[1, \pi_B]_{\mathcal{A}} \circ [e, k]_{\mathcal{A}} = [d, f]_{\mathcal{A}}, \quad [1, \pi_C]_{\mathcal{A}} \circ [e, k]_{\mathcal{A}} = [d, g]_{\mathcal{A}}.
\]
By Lemma~\ref{describe for stab class rel}, there exist morphisms $a, b, a', b' \in \mathcal{A}$ such that the following diagrams commute:
\[
\xymatrix{
&& A^n \times D \ar[lld]_d \ar[rrd]^f && \\
D && A^{n + r + s} \times D \ar[d]_b \ar[u]^a && B \\
&& A^r \times D \ar[ull]^e \ar[urr]_{\pi_B k} &&
}
\quad
\xymatrix{
&& A^n \times D \ar[lld]_d \ar[rrd]^g && \\
D && A^{n + r + s'} \times D \ar[d]_{b'} \ar[u]^{a'} && C \\
&& A^r \times D \ar[ull]^e \ar[urr]_{\pi_C k} &&
}
\]

As before, we may assume $s = s'$, $a = a'$, and $b = b'$. Then the diagram:
\[
\xymatrix{
&& A^n \times D \ar[lld]_d \ar[rrd]^{\langle f, g \rangle} && \\
D && A^{n + r + s} \times D \ar[d]_b \ar[u]^a && B \times C \\
&& A^r \times D \ar[ull]^e \ar[urr]_k &&
}
\]
commutes, and thus $[e, k]_{\mathcal{A}} = [d, h]_{\mathcal{A}}$.
\endproof

So far, we have constructed the category $\mathsf{Span}_\A(\C,\A)$ as a quotient of spans. As mentioned earlier, we will represent the desired indeterminate morphism as a morphism in this category.
The morphism $[!_A,1_A]_\A:1\rightarrow A$ is the indeterminate morphism we are interested in. We denote this morphism by $x$, and we write the category $\mathsf{Span}_\A(\C,\A)$ as $\C[x]$.

Morphisms in $\C[x]$ can be interpreted as polynomials in $x$. The central role of $x$ becomes clearer through a universal property presented in Theorem~\ref{univ_prop}. To prove that theorem, we first state the following proposition. Here, $x^n$ denotes the unique morphism
\[
x \times \cdots \times x : 1 = 1 \times \cdots \times 1 \to A^n = A \times \cdots \times A.
\]

\begin{proposition}\label{xn}
	\begin{itemize}
		\item[(a)] $x^n = [!_{A^n}, 1_{A^n}]$.
		\item[(b)] $x^n \times 1_B = [\pi, 1_{A^n \times B}]$, where $\pi : A^n \times B \to B$ is the projection.
	\end{itemize}
\end{proposition}

\proof
\begin{itemize}
		\item[(a)] For $n = 2$, the uniqueness of $x^2$ in the following commutative diagram implies $x^2 = [!_{A^2}, 1_{A^2}]_\A$:
		\[
		\xymatrix{
		1 \ar[d]_x && 1 \ar[ll] \ar[rr] \ar[d]|{x^2 = [!_{A^2}, 1_{A^2}]_\A} && 1 \ar[d]^x \\
		A && A^2 \ar[rr] \ar[ll] && A
		}
		\]
		By induction on $n$, we obtain $x^n = [!_{A^n}, 1_{A^n}]$.

		\item[(b)] The uniqueness of $x^n \times 1_B$ in the following diagram implies $x^n \times 1_B = [\pi, 1_{A^n \times B}]$:
		\[
		\xymatrix{
		1 \ar[dd]_{x^n} && B \ar[ll] \ar[rr] \ar[dd]|{x^n \times 1_B = [\pi, 1_{A^n \times B}]_\A} && B \ar[dd]^1 \\
		&&&& \\
		A^n && A^n \times B \ar[rr] \ar[ll] && B
		}
		\]
	\end{itemize}
\endproof
The following theorem gives the universal property of $\mathsf{Span}_\A(\C,\A)$ as a category obtained by freely adding an indeterminate morphism.

\begin{theorem}\label{univ_prop}
Let $\mathbf F : \C \to \C'$ be a functor that preserves finite products, and let $a : 1 \to \mathbf F(A)$ be a morphism in $\C'$. Then there exists a unique functor $\mathbf F' : \mathsf{Span}_\A(\C,\A) \to \C'$ such that $\mathbf F'(x) = a$ and the following triangle commutes:
\[
\xymatrix{
\C \ar[rr] \ar[d]_{\mathbf F} && \mathsf{Span}_\A(\C,\A) \ar[lld]^{\mathbf F'} \\
\C' &&
}
\]
\end{theorem}

\proof
Using Proposition~\ref{xn}, a morphism $[p, f]$ with
\(
\xymatrix{
B & A^n \times B \ar[l]_p \ar[r]^f & C
}
\) in \(\C\)
can be written as $[p, f] = [1, f] [p, 1] = [1, f](x^n \times 1_B)$. Based on this, we define:
\[
\mathbf F'[p, f] := \mathbf F(f) \circ (a^n \times 1_{\mathbf F(B)}).
\]

To show that $\mathbf F'$ is well defined, suppose $(p, f) \le_\A (p', f')$ via the following diagram:
\[
\xymatrix{
&& A^n \times B \ar[lld]_p \ar[rrd]^f \ar[dd]_\pi && \\
B && && C \\
&& A^m \times B \ar[ull]^{p'} \ar[urr]_{f'} &&
}
\]
Then we compute:
\[
\mathbf F'[p, f] = \mathbf F(f) \circ (a^n \times 1_{\mathbf F(B)}) = \mathbf F(f') \circ \mathbf F(\pi) \circ (a^n \times 1_{\mathbf F(B)}) = \mathbf F(f') \circ (a^m \times 1_{\mathbf F(B)}) = \mathbf F'[p', f'].
\]

By definition of $\mathbf F'$, we obtain the commutativity of the triangle, as well as the uniqueness.
\endproof

The following theorem shows that the construction of indeterminate morphisms is hereditary. This means that for objects \(A\) and \(B\) in \(\C\), one can first add an indeterminate morphism \(x:1 \to A\) and then add another indeterminate morphism \(y:1 \to B\), or add both of them at once. Before stating the theorem, we define the following classes:
\[
\B = \{ [1,\pi]_\A : B^n \times C \to C \mid [1,\pi]_\A \text{ is a projection in } \mathsf{Span}_\A(\C,\A) \}
\]
\[
\A \circ \B = \{ A^n \times B^m \times C \to C \mid \pi \text{ is a projection in } \C \}
\]

\begin{theorem}
The category \(\mathsf{Span}_\B(\mathsf{Span}_\A(\C,\A),\B)\) is isomorphic to \(\mathsf{Span}_{\A \circ \B}(\C,\A \circ \B)\).
\end{theorem}

\proof
We define the map
\[
[[1,pr]_\A,[p,f]_\A]_\B \longmapsto [pr.p,f]_{\A \circ \B}.
\]
To show that this map is well-defined, suppose
\[
[[1,pr]_\A,[p,f]_\A] \le_\B [[1,pr']_\A,[q,g]_\A]
\]
as shown in the diagram, formed in \(\mathsf{Span}_\A(\C,\A)\):
\[
\xymatrix{
&& B^n \times C \ar[lld]_{[1,pr]_\A} \ar[rrd]^{[p,f]_\A} \ar[dd]_{[1,\pi]_\A} && \\
C && && D \\
&& B^r \times C \ar[ull]^{[1,pr']_\A} \ar[urr]_{[q,g]_\A} &&
}
\]
Then we compute:
\[
\begin{array}{rl}
[pr'.q, g]_{\A \circ \B} &= [q, g]_{\A \circ \B} [pr', 1]_{\A \circ \B} \\
&= [q, g]_{\A \circ \B} [1, \pi]_{\A \circ \B} [\pi, 1]_{\A \circ \B} [pr', 1]_{\A \circ \B} \\
&= [p, f]_{\A \circ \B} [pr, 1]_{\A \circ \B} \\
&= [pr.p, f]_{\A \circ \B}.
\end{array}
\]
So the map is well defined. It is straightforward to check that this map defines an isomorphism of categories.
\endproof

\begin{corollary}\label{prpr_aob}
The functor \(\C \to \mathsf{Span}_{\A \circ \B}(\C,\A \circ \B)\), defined by \(f \mapsto [1, f]_{\A \circ \B}\), preserves finite products.
\end{corollary}

\proof
This functor is the composition of the following functors, each of which preserves finite products. Therefore, the composition also preserves finite products.
\[
\C \to \mathsf{Span}_\A(\C,\A) \to \mathsf{Span}_\B(\mathsf{Span}_\A(\C,\A),\B) \cong \mathsf{Span}_{\A \circ \B}(\C,\A \circ \B)
\]
\endproof

As we have seen, adding indeterminate morphisms \(x:1 \to A\) and \(y:1 \to B\) to the category \(\C\) results in the category \(\mathsf{Span}_{\A \circ \B}(\C,\A \circ \B)\), a quotient of spans. The definition of the compatible system \(\A \circ \B\) suggests a natural way to define a quotient category of spans that includes all such indeterminate morphisms by using a more general compatible system. To achieve this, we use the class of all projections, denoted by \(\Pi\), as a generalization of \(\A \circ \B\). It is straightforward to check that \(\Pi\) is a stable system. Hence, we can form the following quotient category:
\[
\mathsf{Span}_\Pi(\C,\Pi)
\]

\begin{theorem}
The functor \(\mathbf Q : \C \to \mathsf{Span}_\Pi(\C,\Pi)\), defined by \(f \mapsto [1,f]_\Pi\), preserves finite products.
\end{theorem}

\proof
Let
\(
\xymatrix{C && C \times D \ar[ll]_{\pi_1} \ar[rr]^{\pi_2} && D}
\)
be a product diagram in \(\C\). We will show that
\(
\xymatrix{C && C \times D \ar[ll]_{[1,\pi_1]_\Pi} \ar[rr]^{[1,\pi_2]_\Pi} && D}
\)
is a product diagram in \(\mathsf{Span}_\Pi(\C,\Pi)\). Suppose we are given a diagram
\(
\xymatrix{C && E \ar[ll]_{[p,f]_\Pi} \ar[rr]^{[q,g]_\Pi} && D}
\)
with \( (p,f) \) and \( (q,g) \) shown as:
\[
\begin{array}{c@{\hspace{2cm}}c}
\xymatrix{
&& A \times E \ar[lld]_p \ar[rrd]^f && \\
E && && C
}
&
\xymatrix{
&& B \times E \ar[lld]_q \ar[rrd]^g && \\
E && && D
}
\end{array}
\]
in \(\C\). By Corollary~\ref{prpr_aob}, there exists a unique morphism \([r,h]_{\A \circ \B} : E \to C \times D\) such that the triangles in the following diagram commute:
\[
\xymatrix{
C && C \times D \ar[ll]_{[1,\pi_1]_{\A \circ \B}} \ar[rr]^{[1,\pi_2]_{\A \circ \B}} && D \\
&& E \ar[llu]^{[p,f]_{\A \circ \B}} \ar[rru]_{[q,g]_{\A \circ \B}} \ar[u]_{[r,h]_{\A \circ \B}} &&
}
\]
Since \(\A \circ \B \subseteq \Pi\), the same morphism \([r,h]_\Pi\) also makes the following diagram commute:
\[
\xymatrix{
C && C \times D \ar[ll]_{[1,\pi_1]_{\Pi}} \ar[rr]^{[1,\pi_2]_{\Pi}} && D \\
&& E \ar[llu]^{[p,f]_{\Pi}} \ar[rru]_{[q,g]_{\Pi}} \ar[u]_{[r,h]_{\Pi}} &&
}
\]

To prove uniqueness, suppose another morphism \([s,k]_\Pi\) also satisfies:
\[
[1,\pi_1]_\Pi \circ [s,k]_\Pi = [p,f]_\Pi \quad \text{and} \quad [1,\pi_2]_\Pi \circ [s,k]_\Pi = [q,g]_\Pi.
\]
By Lemma~\ref{describe for stab class rel}, there exist projections such that the following diagrams commute:
\[
\xymatrix{&& X\times E\ar[lld]_{s}\ar[rrd]^{\pi_1 k}&&\\
			E&& Y\times X\times A\times E\ar[d]\ar[u]&&C\\
			&&A\times E\ar[ull]^{p}\ar[urr]_{f}&&}\hfil	
		\xymatrix{&& X\times E\ar[lld]_{s}\ar[rrd]^{\pi_2 k}&&\\
			E&& Z\times X\times B\times E\ar[d]\ar[u]&&D\\
			&&B\times E\ar[ull]^{p}\ar[urr]_{f}&&}
\]

Let \(\Pi' = \A \circ \B \circ \X \circ \Y \circ \Z\). An extension of Corollary~\ref{prpr_aob} shows that \([s,k]_{\Pi'} = [r,h]_{\Pi'}\), and so \([s,k]_\Pi = [r,h]_\Pi\). Therefore, \([r,h]_\Pi\) is unique, and the diagram is indeed a product in \(\mathsf{Span}_\Pi(\C,\Pi)\).
\endproof

\begin{theorem}
If $\C$ is a cartesian closed category, then so is $\mathsf{Span}_\Pi(\C,\Pi)$.
\end{theorem}

\proof
We want to show that the exponential object $B^A$ in $\C$ is also an exponential object in $\mathsf{Span}_\Pi(\C,\Pi)$. We do this by showing that the evaluation map $ev: B^A \times A \to B$ in $\C$ induces an evaluation map $[1, ev]_\Pi: B^A \times A \to B$ in $\mathsf{Span}_\Pi(\C,\Pi)$.

Let $[p, f]_\Pi : C \times A \to B$ be a morphism in $\mathsf{Span}_\Pi(\C,\Pi)$, where \((p,f)\) is depicted in \(\C\) as
\(
\xymatrix{
C \times A & D \times C \times A \ar[l]_p \ar[r]^{\ \ \ f} & B
}
\).
There exists a unique morphism $\tilde{f} : D \times C \to B^A$ in $\C$ such that the following diagram commutes:
\[
\xymatrix{
B^A \times A \ar[rr]^{ev} && B \\
(D \times C) \times A \ar[u]^{\tilde{f} \times 1} \ar[urr]_f &&
}
\]
In the following diagrams, the first two (the left and middle ones) are formed in $\C$ using product diagrams. In each of them, the left square is a pullback. This implies that the right diagram, in $\mathsf{Span}_\Pi(\C,\Pi)$, also commutes.
\begin{center}\small
		$\xymatrix{B^A & B^A\times A\ar[l]\ar[r] & A\\
		D\times C\ar[u]^{\tilde f} & D\times C\times A\ar[l]\ar[r]\ar[u]|{\tilde f\times 1} & A\ar[u]^{1}}$
	$\xymatrix{C & C\times A\ar[l]\ar[r] & A\\
		D\times C\ar[u]^{\pi} & D\times C\times A\ar[l]\ar[r]\ar[u]|{\pi\times 1} & A\ar[u]^{1}}$
        $\xymatrix{B^A & B^A\times A\ar[l]\ar[r] & A\\
		C\ar[u]|{[\pi,\tilde f]_\Pi} & C\times A\ar[l]\ar[r]\ar[u]|{[p,\tilde f\times 1]_\Pi} & A\ar[u]^{1}}$
	\end{center}

So we have $[p, \tilde{f} \times 1]_\Pi = [\pi, \tilde{f}]_\Pi \times 1$, and therefore
\[
[1, ev]_\Pi \circ ([\pi, \tilde{f}]_\Pi \times 1) = [p, f]_\Pi.
\]

To prove uniqueness of $[\pi, \tilde{f}]_\Pi$, suppose $[\pi', f']_\Pi$ is another morphism such that
\(
[1, ev]_\Pi \circ ([\pi', f']_\Pi \times 1) = [p, f]_\Pi
\).
By Lemma~\ref{describe for stab class rel}, there exist $r, s \in \Pi$ such that the following diagram commutes:
{\small\[
\xymatrix{
&& D \times C \times A \ar[lldd]_{\pi \times 1} \ar[rd]^{\tilde{f} \times 1} && \\
&&& B^A \times A \ar[rd]^{ev} & \\
C \times A && L \times E \times D \times C \times A \ar[uu]^r \ar[dd]_s && B \\
&&& B^A \times A \ar[ur]_{ev} & \\
&& E \times C \times A \ar[lluu]^{\pi' \times 1} \ar[ru]_{f' \times 1} &&
}
\]}
The projections $r$ and $s$ can be written as
$r = pr \times 1 : (L \times E \times D \times C) \times A \longrightarrow D \times C \times A$
and
$s = pr' \times 1 : (L \times E \times D \times C) \times A \longrightarrow E \times C \times A$.
Then we have
$ev(\tilde{f} pr \times 1) = ev(f' pr' \times 1)$,
which implies
$\tilde{f} pr = f' pr'$.
The commutativity of the following diagram shows that $[\pi, \tilde{f}]_\Pi = [\pi', f']_\Pi$, establishing the uniqueness of $[\pi, \tilde{f}]_\Pi$.
Therefore, $\mathsf{Span}_\Pi(\C,\Pi)$ is cartesian closed.
\[
\xymatrix{
&& D \times C \ar[lld]_{\pi} \ar[rrd]^{\tilde{f}} && \\
C && L \times E \times D \times C \ar[d]_{pr'} \ar[u]^{pr} && B^A \\
&& E \times C \ar[ull]^{\pi'} \ar[urr]_{f'} &&
}
\]
\endproof

For each object $A \in \C$, there is a quotient functor
\(
\mathbf{Q} : \mathsf{Span}_\A(\C,\A) \longrightarrow \mathsf{Span}_\Pi(\C,\Pi)
\)
that maps the morphism \( x = [!_A, 1_A]_\A : 1 \rightarrow A \in \mathsf{Span}_\A(\C,\A) \) to
\( x = [!_A, 1_A]_\Pi : 1 \rightarrow A \in \mathsf{Span}_\Pi(\C,\Pi) \).
This means that $\mathsf{Span}_\Pi(\C,\Pi)$ includes all such indeterminate morphisms.

In what follows, we show that $\mathsf{Span}_\Pi(\C,\Pi)$ has this property in a universal way: it is the colimit of a natural diagram in the category $\mathsf{Cat}$. We build this diagram by collecting all quotient functors of the form
\[
\mathbf{Q} : \mathsf{Span}_{\A_1 \circ \A_2 \circ \dots \circ \A_{n-1}}(\C, \A_1 \circ \A_2 \circ \dots \circ \A_{n-1})
\longrightarrow
\mathsf{Span}_{\A_1 \circ \A_2 \circ \dots \circ \A_n}(\C, \A_1 \circ \A_2 \circ \dots \circ \A_n),
\]
where each $\A_i$ is the compatible system associated to the object $A_i$, for $1 \le i \le n$.

\begin{theorem}
	$\mathsf{Span}_\Pi(\C,\Pi)$ is the colimit of the above diagram.
\end{theorem}

\proof
We first observe that the following diagram forms a natural cocone:
\[
\xymatrix{
\mathsf{Span}_{\A_1 \circ \A_2 \circ \dots \circ \A_{n-1}}(\C, \A_1 \circ \A_2 \circ \dots \circ \A_{n-1})
\ar[rr] \ar[d] &&
\mathsf{Span}_\Pi(\C,\Pi) \\
\mathsf{Span}_{\A_1 \circ \A_2 \circ \dots \circ \A_n}(\C, \A_1 \circ \A_2 \circ \dots \circ \A_n)
\ar[rru] &&
}
\]

Now suppose we are given another cocone to some category $\mathcal{L}$:
\[
\xymatrix{
\mathsf{Span}_{\A_1 \circ \A_2 \circ \dots \circ \A_{n-1}}(\C, \A_1 \circ \A_2 \circ \dots \circ \A_{n-1})
\ar[d] \ar[rrrrrr]^{\mathbf{F}_{\A_1 \circ \A_2 \circ \dots \circ \A_{n-1}}} &&&&&& \mathcal L \\
\mathsf{Span}_{\A_1 \circ \A_2 \circ \dots \circ \A_n}(\C, \A_1 \circ \A_2 \circ \dots \circ \A_n)
\ar[rrrrrru]_{\mathbf{F}_{\A_1 \circ \A_2 \circ \dots \circ \A_n}} &&&&&&
}
\]

We define a functor
\(
\mathbf{U} : \mathsf{Span}_\Pi(\C,\Pi) \longrightarrow \mathcal L
\)
by sending $[\pi, f]_\Pi \mapsto \mathbf{F}_\A[\pi, f]_\A$, where
\(
\xymatrix{B & A \times B \ar[l]_{\pi} \ar[r]^{f} & C}
\)
is a span in $\C$. To check that $\mathbf{U}$ is well defined, suppose we have a commutative diagram where $p$ is a projection:
\[
\xymatrix{
&& A \times B \ar[lld]_{\pi} \ar[rrd]^{f} && \\
B &&&& C \\
&& A \times D \times B \ar[uu]^{p} \ar[llu]_{\pi'} \ar[rru]_{f'} &&
}
\]
This implies $[\pi, f]_{\A \circ \D} = [\pi', f']_{\A \circ \D}$, so by naturality:
\[
\mathbf{U}[\pi, f]_\Pi = \mathbf{F}_\A[\pi, f]_\A
= \mathbf{F}_{\A \circ \D}[\pi, f]_{\A \circ \D}
= \mathbf{F}_{\A \circ \D}[\pi', f']_{\A \circ \D}
= \mathbf{U}[\pi', f']_\Pi.
\]
Hence $\mathbf{U}$ is well defined. The uniqueness of $\mathbf{U}$ is straightforward.
\endproof

\section{Language of a topos}

So far, we have constructed a quotient of spans that contains all indeterminate morphisms in a universal manner. In this section, we show that for a topos \(\T\), the category \(\mathsf{Span}_\Pi(\T,\Pi)\) can be regarded as a coherent system in which the internal language of the topos \(\T\) can be expressed. In our representation of this language, objects of \(\T\) are interpreted as types, and morphisms of the form \([!_A,f]_\Pi:1 \rightarrow B\) are interpreted as terms of type \(B \in \T\).

We denote a term \([!_A,f]_\Pi:1 \rightarrow B\) by \(\phi(x):1 \rightarrow B\), where \(x\) represents \([!_A,1_A]_\Pi:1 \rightarrow A\). Thus, \(x\) is a term of type \(A\), called a variable of type \(A\). Terms of type \(\Omega\) are referred to as formulas.

\begin{definition}
Let \(\alpha(x) = [!_A, f]_\Pi : 1 \rightarrow D\), \(\beta(y) = [!_B, g]_\Pi : 1 \rightarrow D\), and \(\gamma(z) = [!_C, h]_\Pi : 1 \rightarrow \mathbf{P}D\). Then:

\begin{itemize}
    \item \(\alpha(x) = \beta(y)\) is the formula
    \(
    \xymatrix{
        1 \ar[rr]^{\langle \alpha, \beta \rangle} && D \times D \ar[rr]^{[1, \delta_D]_\Pi} && \Omega
    }
    \)

    \item \(\alpha \varepsilon \gamma\) is the formula
    \(
    \xymatrix{
        1 \ar[rr]^{\langle \alpha, \gamma \rangle} && D \times \mathbf{P}D \ar[rr]^{[1, ev]_\Pi} && \Omega
    }
    \)
\end{itemize}

For formulas \(\phi(x) = [!_A, f]: 1 \rightarrow \Omega\) and \(\psi(y) = [!_B, g]: 1 \rightarrow \Omega\), define:

\begin{itemize}
    \item \(\phi \wedge \psi\) as
    \(
    \xymatrix{
        1 \ar[rr]^{\langle \phi, \psi \rangle} && \Omega \times \Omega \ar[rr]^{[1, \wedge]_\Pi} && \Omega
    }
    \)

    \item \(\phi \vee \psi\) as
    \(
    \xymatrix{
        1 \ar[rr]^{\langle \phi, \psi \rangle} && \Omega \times \Omega \ar[rr]^{[1, \vee]_\Pi} && \Omega
    }
    \)

    \item \(\phi \implies \psi\) as
    \(
    \xymatrix{
        1 \ar[rr]^{\langle \phi, \psi \rangle} && \Omega \times \Omega \ar[rr]^{[1, \implies]_\Pi} && \Omega
    }
    \)

    \item \(\mathrm{not}\ \phi\) as
    \(
    \xymatrix{
        1 \ar[rr]^{\phi} && \Omega \ar[rr]^{[1, \mathrm{not}]_\Pi} && \Omega
    }
    \)

    \item \(\forall \phi(x) = [1, \forall_A \tilde{f}]\)

    \item \(\exists \phi(x) = [1, \exists_A \tilde{f}]\)

    where \(\forall_A\) is the right adjoint and \(\exists_A\) is the left adjoint to \(\mathbf{P}(!_A): \Omega \rightarrow \mathbf{P}(A)\), and \(\tilde{f}\) is obtained from the diagram:
    \[
    \xymatrix{
        \mathbf{P}(A) \times A \ar[rr]^{ev} && \Omega \\
        1 \times A \ar[u]^{\tilde{f} \times 1} \ar[rr] && A \ar[u]_{f}
    }
    \]

    \item For \(\phi(x) = [1, f][!_A, 1]: 1 \rightarrow A \rightarrow \Omega\), the expression \(\{x \in A: \phi(x)\}\) is defined as the unique morphism obtained from the diagram:
    \[
    \xymatrix{
        \mathbf{P}A \times A \ar[rr]^{[1, ev]_\Pi} && \Omega \\
        1 \times A \ar[u]^{\{x \in A : \phi(x)\} \times 1} \ar[urr]_{[1, f]_\Pi} &&
    }
    \]
\end{itemize}
\end{definition}

\begin{proposition}
\(\forall \phi(x)\) and \(\exists \phi(x)\) are well defined.
\end{proposition}

\proof
Let the left diagram below be given with \(\pi \in \Pi\). We aim to show \(\forall_A \tilde{f} = \forall_{A \times C} \widetilde{f\pi}\). The right diagram implies \(\widetilde{f\pi} = \mathbf{P}(\pi) \tilde{f}\).

\[
\xymatrix{
&& A \times C \ar[lld] \ar[rrd]^{f\pi} \ar[dd]_{\pi} && \\
1 &&&& \Omega \\
&& A \ar[llu] \ar[rru]_{f} &&
}
\quad
\xymatrix{
\mathbf{P}(A \times C) \times (A \times C) \ar[rr] && \Omega \\
\mathbf{P}A \times (A \times C) \ar[u]_{\mathbf{P}(\pi) \times 1} \ar[rr]^{1 \times \pi} && \mathbf{P}A \times A \ar[u]_{ev} \\
1 \times (A \times C) \ar[u]_{\tilde{f} \times 1} \ar[rr]_{1 \times \pi} && 1 \times A \ar[u]_{\tilde{f} \times 1} \ar@/_2.5pc/[uu]_{f}
}
\]

The adjunction diagram below, together with \(!_{{A \times C}} = !_A \pi\), implies \(\forall_{A \times C} = \forall_A \forall_\pi\).

\[
\xymatrix{
\mathbf{P}(A \times C) \ar@<1ex>[rr]^{\forall_\pi} && \mathbf{P}A \ar@<1ex>[ll]^{\mathbf{P}(\pi)} \ar@<1ex>[rr]^{\forall_A} && \Omega \ar@<1ex>[ll]^{\mathbf{P}(!_A)}
}
\]

The following pullback and pullback-complement squares illustrate the external forms of \(\mathbf{P} \pi\) and \(\forall_\pi\), respectively. From the right square, we get \(\pi^{-1} d = d \times 1\), which implies \(\forall_\pi \mathbf{P} \pi = 1\). Therefore,
\(
\forall_{A \times C} \widetilde{f\pi} = \forall_A \forall_\pi \mathbf{P} \pi \tilde{f} = \forall_A \tilde{f},
\)
so \(\forall \phi(x)\) is well defined.

\[
\xymatrix{
D \times C \ar[rr]^{d \times 1} \ar[d] \ar@{}[rrd]|{p.b.} && A \times C \ar[d]^{\pi} \\
D \ar[rr]_{d} && A \times 1
}
\quad
\xymatrix{
D \times C \ar[rr]^{d \times 1} \ar[d]_{1 \times !_C} \ar@{}[rrd]|{p.b.c.} && A \times C \ar[d]^{1 \times !_C = \pi} \\
D \times 1 \ar[rr]_{d \times 1} && A \times 1
}
\]

Using \cite[Lemma 2.3.6]{john}, we obtain \(\exists_C \mathbf{P} \pi = 1\). Hence,
\(
\exists_{A \times C} \widetilde{f\pi} = \exists_A \exists_C \mathbf{P} \pi \tilde{f} = \exists_A \tilde{f},
\)
so \(\exists \phi(x)\) is also well defined.
\endproof

\section{Logical relations on span categories}

In \cite{hsty}, compatible relations on span categories, in which their quotients are allegories, are studied, and it is shown that for a pullback stable factorization system \((\E,\M)\) in a finitely complete category \(\C\), \(\mathsf{Rel}(\C,\E,\M)\cong \mathsf{Span}_\E(\C)\) \cite[Theorem 2.3]{hsty}.
For a regular category \(\C\), with \(\E=\mathbf{RegEpi}(\C)\), \(\M=\mathbf{Mono}(\C)\), it is well known that \(\mathsf{Rel}(\C,\E,\M)\cong \mathsf{Span}_\E(\C)\) is a tabular allegory and \(\mathsf{Map}(\mathsf{Span}_\E(\C))\cong \C\). This motivates us to investigate which quotients of \(\mathsf{Span}(\T)\), for a topos \(\T\), are toposes.
In \cite{john}, it is shown that maps of a power allegory form a topos. Inspired by this, we investigate conditions on a compatible relation \(\sim\) to make \(\mathsf{Span}_\sim(\T)\) a power allegory.

Allegories were presented for the first time in \cite{freyd} as categories which reflect properties that hold in the category of relations.

\begin{definition}
	An \emph{allegory} is a locally ordered $2$-category $\A$ whose hom-posets
	have binary intersections, equipped with an anti-involution $\phi \mapsto \phi^{\circ}$ and satisfying the modular law
	\[
	\psi\phi \cap \chi \leq (\psi \cap \chi\phi^{\circ})\phi,
	\]
	whenever this makes sense.
\end{definition}
\begin{definition}
In an allegory, a morphism is called \emph{map} if \(1\le r^\circ \cdot r\) and \(r\cdot r^\circ\le 1\). The subcategory of maps of an allegory \(\A\) is denoted by $\mathsf{MAP}(\A)$.
\end{definition}
A power allegory is a division allegory with some extra properties. First, we give the definition of a division allegory and then the definition of a power allegory. See \cite{john} for more information.
\begin{definition}[\cite{john}]
	An allegory $A$ is called a \emph{division allegory} if, for each $\phi:A\rightarrow B$ and object $C$, the order preserving map $(-)\phi:\A(B,C)\longrightarrow \A(A,C)$ has a right adjoint, which we call right division by $\phi$ and denote $(-)/\phi$.
\end{definition}
Of course, the anti-involution ensures that if we have right division we also have left division $\phi\setminus (-)$ (right adjoint to $\phi(-)$).
We write $(\phi|\psi)$ for
$$(\phi\setminus \psi)\cap (\psi\setminus \phi)^\circ.$$
\begin{definition}[\cite{john}]
	A division allegory $\A$ is called a \emph{power allegory} if there is an operation assigning to each object $A$ a morphism $\in_A:PA\rightarrow A$ satisfying $(\in_A|\in_A)=1_{PA}$ and
	$$1_B\le (\phi\setminus\in_A)(\in_A\setminus\phi)$$
	for any $\phi:B\rightarrow A$.
\end{definition}

Every topos has an \((\mathbf{Epi},\mathbf{Mono})\) factorization. In the following, we denote a topos as \(\T\) and its epi-mono factorization as \((\E,\M)\). Utilizing \((\E,\M)\), we define a kind of compatible relation such that the quotient arising from it will be shown to be a power allegory. 

\begin{definition}
	For a topos $\T$, a compatible relation $\sim$ on $\mathsf{Span}(\T)$ is called \emph{logical} if:
	\begin{itemize}
		\item
		$\E\subseteq \sim$
		\item
		for spans $(f,g),(h,k):A\rightarrow C$ and a morphism $a:A\rightarrow B$
		\[
		(f,g)\sim(h,k) \implies (\pi_1\forall_{a\times 1}m,\pi_2\forall_{a\times 1}m)\sim(\pi_1\forall_{a\times 1}n,\pi_2\forall_{a\times 1}n)
		\]
		where \( m \) and \( n \) are the \( \mathcal{M} \)-parts of the morphisms \( \langle f, g \rangle : D \to A \times C \) and \( \langle h, k \rangle : D' \to A \times C \), respectively, where \( \langle f, g \rangle \) and \( \langle h, k \rangle \) denote the unique morphisms in \(\T\) induced by the universal property of the product, and \( D \) and \( D' \) are the domains of \( f \) and \( h \), respectively\footnote{We use angle brackets to denote the unique morphism resulting from a product diagram. Note that these morphisms are not spans.}.

	\end{itemize}
\end{definition}

First, we show that $\mathsf{Span}_\sim(\T)$ is a division allegory, for a logical relation $\sim$. Since $\E\subseteq \sim$, the mapping $Q:\mathsf{Span}_\E(\T)\longrightarrow \mathsf{Span}_\sim(\T)$, defined by \(Q([f,g]_\E)=[f,g]_\sim\), is a representation of allegories, meaning that $Q$ preserves $^\circ$ and $\cap$.

\begin{theorem}
	For a topos $\T$, $\mathsf{Span}_\E(\T)$ is a division allegory.
\end{theorem}

\proof
	Utilizing \cite[Theorem 3.4.2]{john} and \cite[Theorem 4.2]{hsty}, $\mathsf{Span}_\E(\T)$ is a division allegory, where \([h,k]_\E/[f,g]_\E := [\pi_1a,\pi_2a]_\E\), in which $a=\forall_{g\times 1}(f\times 1)^*(m_{\langle h,k\rangle})$, and $m_{\langle h,k\rangle}$ is the mono part of $\langle h,k\rangle$.
\endproof

\begin{theorem}
	For a logical relation $\sim$, $\mathsf{Span}_\sim(\T)$ is a division allegory and
	\[
	Q((-)/[f,g]_\E)=(-)/[f,g]_\sim.
	\]
\end{theorem}

\proof
Let $(-)/[f,g]_\sim := Q((-)/[f,g]_\E)$. It follows from the definition of logical relation that this definition is well-defined. We have
	\[
	[h,k]_\sim = Q [h,k]_\E \le Q(([h,k]_\E [f,g]_\E)/[f,g]_\E) = ([h,k]_\sim [f,g]_\sim)/[f,g]_\sim
	\]
and
	\[
	([r,s]_\sim/[f,g]_\sim)[f,g]_\sim = Q([r,s]_\E/[f,g]_\E) Q[f,g]_\E
	\]
    \[= Q(([r,s]_\E/[f,g]_\E)[f,g]_\E) \le Q[r,s]_\E = [r,s]_\sim.
	\]
Therefore, $(-)/[f,g]_\sim$ is right adjoint to $(-)[f,g]_\sim$.
\endproof

\begin{theorem}
	$\mathsf{Span}_\E(\T)$ is a power allegory.
\end{theorem}

\proof
Since $\T$ is a topos, $\mathsf{Rel}(\T,\E,\M)$ is a power allegory with $\in_A:PA \rightarrow A$ defined as
\[
\xymatrix{
	& \in_A \ar[ld] \ar[rd] \ar[d] & \\
	PA & PA \times A \ar[l] \ar[r] & A
}
\]
By \cite[Theorem 2.3]{hsty}, we have $\mathsf{Rel}(\T,\E,\M) \cong \mathsf{Span}_\E(\T)$, and $\in_A:PA \rightarrow A$ in $\mathsf{Span}_\E(\T)$ is defined as in $\mathsf{Rel}(\T,\E,\M)$.
\endproof

\begin{theorem}
	For a logical relation $\sim$, $\mathsf{Span}_\sim(\T)$ is a power allegory and $\mathsf{Map}(\mathsf{Span}_\sim(\T))$ is a topos.
\end{theorem}

\proof
Let $\in_A:PA \rightarrow A$ in $\mathsf{Span}_\sim(\T)$ be $Q(\in_A:PA \rightarrow A)$. Since $Q((-)/[f,g]_\E) = (-)/[f,g]_\sim$ and $Q$ is a representation, $\mathsf{Span}_\sim(\T)$ is a power allegory. Then by \cite[Corollary 3.4.7]{john}, $\mathsf{Map}(\mathsf{Span}_\sim(\T))$ is a topos.
\endproof

Denoting \(\mathsf{Map}(Q)\) by \(\eta\):
\begin{corollary}\label{logical functor}
    $\eta:\T \longrightarrow \mathsf{Map}(\mathsf{Span}_\sim(\T))$ is a logical functor.
\end{corollary}


In the rest of this section, we present a different kind of compatible relation, generated by a class of \textit{endospans}, that can be considered as an extension of relations generated by classes of morphisms. Utilizing them, we can generate some logical relations.
An endospan, as depicted below, is a span in which its domain and codomain are the same.
\begin{center}
	$\xymatrix{&& D\ar[lld]_{p}\ar[rrd]^{q}&&\\
 	A&&&&A}$
\end{center}
If in the above endospan $p=q$ and $p$ is an isomorphism, it is called an endospan of an iso.

\begin{definition}
\begin{itemize}
\item A class of endospans is called \emph{saturated} if it contains all endospans of isos.
\item Suppose $\A$ is a saturated class of endospans. The compatible relation generated by $\A$, denoted by $\sim_\A$, is defined to be the smallest compatible relation $\sim$ on the category $\mathsf{Span}(\C)$ such that for all $(a,b)$ in $\A$, \((a,b)\sim (1,1)\).
\end{itemize}
\end{definition}
In the next  proposition, we  explain how this smallest relation is constructed and provide a concrete representation of it.
\begin{proposition}
For a saturated class of endospans, $\A$, the compatible relation generated by $\A$ is described as follows:

$(h,k)\sim (r,s) \iff$ there are decompositions
$(h,k)=(h_n,k_n)\cdots(h_1,k_1)$ and $(r,s)=(r_m,s_m)\cdots(r_1,s_1)$, and endospans
$(a_1,b_1),\cdots,(a_n,b_n)\in\A$ and $(c_1,d_1),\cdots,(c_m,d_m)\in\A$ such that:
\[
(r_m,s_m)(c_m,d_m)\cdots(r_1,s_1)(c_1,d_1)=(h_n,k_n)(a_n,b_n)\cdots(h_1,k_1)(a_1,b_1)
\]
\end{proposition}

\proof
Obvious.
\endproof

\begin{example}
\begin{itemize}
\item Let $\I$ be the class of all endospans of isos. The compatible relation generated by this class is defined as follows: $(f,g)\sim(h,k)$ if there is an isomorphism $\phi$ such that $f=h\phi$ and $g=k\phi$. So $\mathsf{Span}_{\sim}(\C)$ is the ordinary category of spans.
\item For a stable system of morphisms $\B$, we can form a saturated class of endospans containing $(b,b)$ for all $b\in \B$. The compatible relation generated by this class of endospans is equivalent to $\sim_\B$.
\item For a morphism $f:A\rightarrow B$, we can form a saturated endospan class by adding the kernel pair of $f$ to $\I$, the class of all endospans of isos.
\item For a morphism $f:A\rightarrow B$, a saturated class of endospans can be formed by adding the kernel pair of $f$ to the class of endospans containing $(e,e)$ for epimorphisms $e$.
\end{itemize}
\end{example}

\begin{definition}
For a morphism $f:A\rightarrow B$, we define $K(f)$ to be the saturated class of endospans containing the kernel pairs of all morphisms $h$ in which $f=gh$ for some morphism $g$, and $(e,e)$ for all epimorphisms $e$.
\end{definition}

The compatible relations generated by $K(f)$ imply $(p_1,p_2)\sim_{K(f)}(1,1)$, in which $p_1,p_2$ are obtained by the following pullback diagram, where $f=gh$ for some morphism $g$:
\[
\xymatrix{&&P\ar[lld]_{p_1}\ar[rrd]^{p_2}&&\\
	A\ar[rrd]_{h}&&&&A\ar[lld]^{h}\\
&&B}
\]

\begin{lemma}
Using the above definitions and notations, we have:
\begin{itemize}
\item[(a)] For an epimorphism $e$, $[1,e]_{K(e)}$ is an isomorphism and its inverse is $[e,1]_{K(e)}$.
\item[(b)] If $f=gh$, then $K(h)\subseteq K(f)$.
\end{itemize}
\end{lemma}

\proof
Obvious.
\endproof

The smallest logical relation containing $K(f)$ is denoted by $L(f)$.

\begin{lemma}\label{forall}
The following diagram is formed by pullbacks, in which $g$ is an epimorphism. Then $\forall_{g\times g}\langle q_1,q_2\rangle=\langle p_1,p_2\rangle$.
\[
\xymatrix{
Q\ar@/_1pc/[dd]_{q_1}\ar@/^1pc/[rr]^{q_2}\ar[r]_{v_2}\ar[d]^{v_1}&R\ar[r]^{r}\ar[d]_{r_2}&C\ar[d]^{g}\\
R\ar[r]^{r_1}\ar[d]_{r}&P\ar[r]^{p_2}\ar[d]_{p_1}&B\ar[d]^{f}\\
C\ar[r]_{g}&B\ar[r]_{f}&A}
\]
\end{lemma}

\proof
Let $(g\times g)^{-1}\langle x,y\rangle \le\langle q_1,q_2\rangle$. Then there is an arrow $i$ such that
$(g\times g)^{-1}\langle x,y\rangle=\langle q_1,q_2\rangle i$. Set $(g\times g)^{-1}\langle x,y\rangle=\langle x',y'\rangle$ and $\langle x,y\rangle^{-1}(g\times g)=e$.
Since $g$ is an epimorphism, $(g \times g)$ is also an epimorphism, and since epimorphisms are stable under pullbacks in a topos, it follows that $e$ is an epimorphism as well.

We have the following equalizer diagrams:
\begin{center}
$\xymatrix{P\ar[rr]^{\langle p_1,p_2\rangle}&&B\times B\ar@<1ex>[rr]^{f\pi_1}\ar@<-1ex>[rr]_{f\pi_2}&&A}$
\hfil
$\xymatrix{P\ar[rr]^{\langle q_1,q_2\rangle}&&C\times C\ar@<1ex>[rr]^{fg\pi'_1}\ar@<-1ex>[rr]_{fg\pi'_2}&&A}$
\end{center}

We have $fg\pi'_1=f\pi_1(g\times g)$ and $fg\pi'_2=f\pi_2(g\times g)$. So:
\[
\begin{array}{ll}
fxe &=f\pi_1\langle x,y\rangle e\\
&=f\pi_1(g\times g)\langle x',y'\rangle\\
&=fg\pi'_1\langle q_1,q_2\rangle i\\
&=fg\pi'_2\langle q_1,q_2\rangle i\\
&=f\pi_2(g\times g)\langle x',y'\rangle\\
&=f\pi_2\langle x,y\rangle e\\
&=fye
\end{array}
\]

Since $e$ is an epimorphism, $fx=fy$. Hence $\langle x,y\rangle\le\langle p_1,p_2\rangle$.
By the following pullback diagrams we get $(g\times g)^{-1}\langle p_1,p_2\rangle=\langle q_1,q_2\rangle$. Then, $\langle x,y\rangle\le\langle p_1,p_2\rangle$ implies $(g\times g)^{-1}\langle x,y\rangle \le\langle q_1,q_2\rangle$.

\begin{center}
$\xymatrix{Q\ar[rr]^{v_1}\ar[d]_{\langle q_1,q_2\rangle}&&R\ar[rr]^{r_1}\ar[d]_{\langle r,p_2r_1\rangle}&&P\ar[d]^{\langle p_1,p_2\rangle}\\
C\times C\ar[rr]_{1\times g} && C\times B\ar[rr]_{g\times 1}&&B\times B}$
\end{center}
\endproof

\begin{corollary}
Suppose $f = gh$ with $h$ an epimorphism. Then $L(g) \subseteq L(f)$.
\end{corollary}

\proof
Let $g=uv$ and consequently $f=uvh$. So the kernel pair of $v$ is related to $(1,1)$ by $L(g)$ and the kernel pair of $vh$ is related to $(1,1)$ by $L(f)$. Since the kernel pairs of $h$ and $vh$ are related by $L(f)$, by using $\forall_{h\times h}$ and Lemma \ref{forall}, the kernel pair of $v$ is related to $(1,1)$ by $L(f)$.
\endproof

\section{Booleanization of a topos}

In this section, we aim to associate a Boolean topos to each topos in a universal way. To achieve this, we introduce a class of morphisms called logical classes, which support certain logical operations. Using this, we construct a quotient of spans, yielding the associated Boolean topos.

\begin{definition}
Let $\T$ be a topos and let $\W$ be a class of morphisms in $\T$. We call $\W$ a \emph{logical class} if it satisfies the following conditions:
\begin{itemize}
	\item
	$\W$ is closed under composition, pullbacks, and contains all isomorphisms,
	\item
	$\E \subseteq \W$,
	\item
	for each $w \in \W$, its $\M$-part is also in $\W$,
	\item
	for any monomorphism $m \in \W$, and for any monomorphism $f$ and morphism $g$ in $\T$, the morphism $\forall_g m$ is also in $\W$:
	\[
	\xymatrix{
	\ar[r]^m & \ar[r]^f & \ar[d]^g \\
	\ar[r]_{\forall_g m} & \ar[r]_{\forall_g f} &
	}
	\]
\end{itemize}
\end{definition}

\begin{theorem}
If $\W$ is a logical class, then $\sim_\W$ is a logical relation.
\end{theorem}

\proof
It can be easily verified.
\endproof

In any topos, the morphism $b: 1+1 \rightarrow \Omega$ is a monomorphism. Our goal is to make this morphism an isomorphism. Let $\B(\T)$ be the smallest logical class containing $b:1+1 \rightarrow \Omega$.

\begin{theorem}
$\mathsf{Map(Span}_{\B(\T)}(\T))$ is a Boolean topos.
\end{theorem}

\proof
Since $(b,b) \sim_{\B(\T)} (1,1)$, we have $[b,b]_{\B(\T)} = [1,1]_{\B(\T)}$. Thus, $[1,b]_{\B(\T)}$ is a retraction. Because $b$ is mono, we get $[b,1]_{\B(\T)}[1,b]_{\B(\T)} = 1$. Hence, $[1,b]$ is an isomorphism in $\mathsf{Span}_{\B(\T)}(\T)$, and therefore also in $\mathsf{Map(Span}_{\B(\T)}(\T))$.

By Theorem \ref{logical functor}, the functor $\eta:\T \to \mathsf{Map(Span}_{\B(\T)}(\T))$ is logical. By \cite[Corollary 2.2.10]{john}, $\eta$ is cocartesian. Thus,
\[
\eta(b:1+1 \rightarrow \Omega) = [1,b]:1+1 \rightarrow \Omega.
\]
So $\mathsf{Map(Span}_{\B(\T)}(\T))$ is a Boolean topos, as required.
\endproof

We now show that this construction is universal.

\begin{lemma}\label{FB(B"}
For any logical functor $F:\T \to \T'$, we have $F(\B(\T)) \subseteq \B(\T')$.
\end{lemma}

\proof
Since $F$ preserves epis, monos, and $\forall$, one can easily check that $F^{-1}(\B(\T'))$ is a logical class. Because $F(b) = b'$, we have $b \in F^{-1}(\B(\T'))$, and thus $\B(\T) \subseteq F^{-1}(\B(\T'))$. Therefore, $F(\B(\T)) \subseteq \B(\T')$.
\endproof

\begin{theorem}\label{logical functor gives logical functor}
Let $F:\T \to \T'$ be a logical functor.
\begin{itemize}
	\item[(a)] The map $PF:\mathsf{Span}_{\B(\T)}(\T) \to \mathsf{Span}_{\B(\T')}(\T')$ defined by $[f,g]_{\B(\T)} \mapsto [Ff,Fg]_{\B(\T')}$ is a representation of allegories.
	\item[(b)] $\mathsf{Map}(PF): \mathsf{Map(Span}_{\B(\T)}(\T)) \to \mathsf{Map(Span}_{\B(\T')}(\T'))$ is a logical functor.
\end{itemize}
\end{theorem}

\proof
For (a), Lemma \ref{FB(B"} ensures the map is well-defined. The rest follows from the fact that $F$ preserves pullbacks. For (b), the result follows from the definition of $\in_A$ in both allegories.
\endproof

Let $\mathsf{BoolTop}$ denote the category whose objects are Boolean toposes and whose morphisms are logical functors. This forms a subcategory of the category $\mathsf{Top}$ of toposes and logical functors. Using Theorem \ref{logical functor gives logical functor}, we define the functor
\[
\mathbf{Bool}:\mathsf{Top} \to \mathsf{BoolTop}
\]
where $\mathbf{Bool}(F)$ and $\mathbf{Bool}(\T)$ denote $\mathsf{Map}(PF)$ and $\mathsf{Map(Span}_{\B(\T)}(\T))$, respectively.

\begin{theorem}
$\mathsf{BoolTop}$ is a reflective subcategory of $\mathsf{Top}$.
\end{theorem}

\proof
We show that the functor $\mathbf{Bool}$ is left adjoint to the inclusion functor $\iota: \mathsf{BoolTop} \to \mathsf{Top}$. It suffices to show that
\[
\eta: \T \to \mathsf{Map(Span}_{\B(\T)}(\T)) = \iota \cdot \mathbf{Bool}(\T)
\]
is universal.

Let $F: \T \to \T' = \iota(\T')$ be a logical functor. Since $F$ is cocartesian and $\T'$ is a Boolean topos, $F(b)$ is an isomorphism. Theorem \ref{logical functor gives logical functor} yields the functor $\mathsf{Bool}(F): \mathsf{Bool}(\T) \to \mathsf{Bool}(\T')$. It is easy to check that $\B(\T') = \E$, hence $\mathsf{Bool}(\T') = \T'$. So we have the commutative triangle:
\[
\xymatrix{
\T \ar[rrrr]^{\eta} \ar[d]_F &&&& \mathsf{Map(Span}_{\B(\T)}(\T)) = \iota \cdot \mathsf{Bool}(\T) \ar[lllld]^{\mathsf{Bool}(F)} \\
\T' &&&&
}
\]

For uniqueness, let $[f,g]_{\B(\T)}$ be a map in $\mathsf{Span}_{\B(\T)}(\T)$. Then $[f,1]_{\B(\T)}$ is an isomorphism with inverse $[1,f]_{\B(\T)}$. For any functor $G$ such that $G \circ \eta = F$, we compute:
\[
\begin{array}{ll}
G[f,g]_{\B(\T)} &= G[1,g]_{\B(\T)} \cdot G[f,1]_{\B(\T)} \\
&= G[1,g]_{\B(\T)} \cdot (G[1,f]_{\B(\T)})^{-1} \\
&= G\eta(g) \cdot (G\eta(f))^{-1} \\
&= F(g) \cdot F(f)^{-1} \\
&= \mathsf{Bool}(F)[f,g]_{\B(\T)}.
\end{array}
\]
This completes the proof.
\endproof

\refs
\bibitem [Freyd, 1990]{freyd} P. J. Freyd, and A. Scedrov, Categories, allegories. North-Holland Mathematical Library, 39. North-Holland Publishing Co., Amsterdam, 1990. xviii+296 pp. ISBN: 0-444-70368-3; 0-444-70367-5.

\bibitem [Hosseini, 2020]{hoshitho} S. N. Hosseini, A. R. Shir Ali Nasab, and W. Tholen, Fraction, restriction, and range categories from stable systems of morphisms. J. Pure Appl. Algebra 224 (2020), no. 9, 106361, 28 pp.

\bibitem [Hosseini, 2022]{hsty}  S. N. Hosseini, A. R. Shir Ali Nasab, W. Tholen, and L. Yeganeh, Quotients of span categories that are allegories and the representation of regular categories.  Appl Categor Struct 30, 1177–1201 (2022). https://doi.org/10.1007/s10485-022-09687-9

\bibitem [Johnstone, 2002]{john} P. T. Johnstone, Sketches of an elephant: a topos theory compendium. Vol. 1. Oxford Logic Guides, 43. The Clarendon Press, Oxford University Press, New York, 2002. xxii+468+71 pp. ISBN: 0-19-853425-6

\bibitem [Lambek, 1986]{lambscot} J. Lambek, and P. J. Scott, Introduction to higher order categorical logic. Cambridge Studies in Advanced Mathematics, 7. Cambridge University Press, Cambridge, 1986. x+293 pp. ISBN: 0-521-24665-2

\bibitem [Mac Lane, 1992]{macmor} S. Mac Lane, and I. Moerdijk, Sheaves in geometry and logic. A first introduction to topos theory. Corrected reprint of the 1992 edition. Universitext. Springer-Verlag, New York, 1994. xii+629 pp. ISBN: 0-387-97710-4.

\endrefs

\end{document}